\documentclass[arxiv,locbib]{myarticle}

\newcommand{\tLeq}{\eq{\tL}}
\DeclareMathOperator{\oppop}{opp}
\newcommand{\opp}[1]{#1^{\oppop}}

\DeclareMathOperator{\Ldiveq}{\LL^{\eqop}_{\divop}}
\DeclareMathOperator{\LDdiveq}{\LL^{\eqop}_{\D,\divop}}

\initbibli

\author{Silvain Rideau\thanks{Partially supported by ValCoMo (ANR-13-BS01-0006)} and Pierre Simon\footnotemark[1]}
\title{Definable and invariant types in enrichments of $\NIP$ theories}
\date{\today}

\begin{document}
\maketitle

\begin{abstract}
Let $T$ be an $\NIP$ $\LL$-theory and $\tT$ be an enrichment. We give a sufficient condition on $\tT$ for the underlying $\LL$-type of any definable (respectively invariant) type over a model of $\tT$ to be definable (respectively invariant). These results are then applied to Scanlon's model completion of valued differential fields.
\end{abstract}

Let $T$ be a theory in a language $\LL$ and consider an expansion $T\subseteq \tT$ in a language $\tL$. In this paper, we wish to study how invariance and definability of types in $T$ relate to invariance and definability of types in $\tT$. More precisely, let $\fU\models\tT$ be a monster model and consider some type $\widetilde{p}\in \TP(\fU)$ which is invariant over some small $M\models\tT$. Then the reduct $p$ of $\widetilde{p}$ to $\LL$ is of course invariant under the action of the $\tL$-automorphisms of $\fU$ that fix $M$ (which we will denote as $\tL(M)$-invariant), but there is, in general, no reason for it to be $\LL(M)$-invariant. Similarly, if $\widetilde{p}$ is $\tL(M)$-definable, $p$ might not be $\LL(M)$-definable.

When $T$ is stable, and $\phi(x;y)$ is an $\LL$-formula, $\phi$-types are definable by Boolean combinations of instances of $\phi$. It follows that if $\widetilde{p}$ is $\tL(M)$-invariant then $p$ is both $\LL(M)$-invariant and $\LL(M)$-definable. Nevertheless, when $T$ is only assumed to be $\NIP$, then this is not always the case. For example one can take $T$ to be the theory of dense linear orders and $\tL=\{\leq , P(x)\}$ where $P(x)$ is a new unary predicate naming a convex non-definable subset of the universe. Then there is a definable type in $\tT$ lying at some extremity of this convex set whose reduct to $\LL=\{\leq\}$ is not definable without the predicate.

In the first section of this paper, we give a sufficient condition (in the case where $T$ is $\NIP$) to ensure that any $\tL(M)$-invariant (resp. definable) $\LL$-type $p$ is also $\LL(M)$-invariant (resp. definable). The condition is that there exists a model $M$ of $\tT$ whose reduct to $\LL$ is uniformly stably embedded in every elementary extension of itself. In the case where $T$ is o-minimal for example, this happens whenever the ordering on $M$ is complete.

The main technical tool developed in this first section is the notion of external separability (Definition\,\ref{def:ext sep}). Two sets $X$ and $Y$ are said to be externally separable if there exists an externally definable set $Z$ such that $X\subseteq Z$ and $Y\cap Z = \emptyset$. In Proposition\,\ref{prop:ext sep is def}, we show that in $\NIP$ theory, external separability is essentially a first order property. The results about definable and invariant sets then follow by standard methods along with a "local representation" of $\phi$-types from \cite{SimInvNIP}.


The motivation for these results comes from the study of expansions of $\ACVF$ and in particular the model completion $\VDF$ defined by Scanlon \cite{ScaDValF} of valued differential fields with a contractive derivation, i.e. a derivation $\D$ such that for all $x$, $\val(\D(x))\geq\val(x)$. In the third section, we deduce, from the previous abstract results, a characterisation of definable (resp. invariant) types in models of $\VDF$ in terms of the definability (resp. invariance) of the underlying $\ACVF$-type. This characterisation also allows us to control the canonical basis of definable types in $\VDF$, an essential step in proving elimination of imaginaries for that theory in \cite{RidVDF}.

\subsection{Notation}

Let us now define some notation that will be used throughout the paper. When $\phi(x;y)$ is a formula, we implicitly consider that $y$ is a parameter of the formula and we define $\opp{\phi}(y;x)$ to be equal to $\phi(x;y)$.

We write $N\prec^+ M$ to denote that $M$ is a $\card{N}^{+}$-saturated and (strongly) $\card{N}^{+}$-homogeneous elementary extension of $N$.

Let $X$ be an $\LL(M)$-definable set (or a union of definable sets) and $A\subseteq M$. We denote by $X(A)$ the set of realisations of $X$ in $A$, i.e. the set $\{a\in A\mid M\models a\in X\}$. If $\Real$ is a set of definable sets (in particular a set of sorts), we define $\Real(A) := \bigcup_{R\in\Real} R(A)$.


\section{External separability}
\label{sec:ext sep}

\begin{definition}[ext def](Externally $\phi$-definable)
Let M be an $\LL$-structure, $\phi(x;t)$ be an $\LL$-formula and $X$ a subset of some Cartesian power of $M$.  We say that $X$ is externally $\phi$-definable if there exist $N\supsel M$ and a tuple $a\in N$ such that $X= \phi(M;a)$.
\end{definition}

\begin{definition}[ext sep](Externally $\phi$-separable)
Let M be an $\LL$-structure, $\phi(x;t)$ be an $\LL$-formula and $X$, $Y$ be subsets of some Cartesian power of $M$.  We say that $X$ and $Y$ are externally $\phi$-separable if there exist $N\supsel M$ and a tuple $a\in N$ such that $X\subseteq \phi(M;a)$ and $Y\cap\phi(M;a) = \emptyset$.
\end{definition}

We will say that $X$ and $Y$ are $\phi$-separable if $a$ can be chosen in $M$. Note that a set $X$ is externally $\phi$-definable if $X$ and its complement are externally $\phi$-separable.

\begin{proposition}[ext sep is def]
Let $T$ be an $\LL$-theory and $\phi(x;t)$ an $\NIP$ $\LL$-formula. Let $U(x)$ and $V(x)$ be new predicate symbols and let $\LL_{U,V} := \LL\cup\{U,V\}$. Then, there is an $\LL_{U,V}$-sentence $\theta_{U,V}$ and an $\LL$-formula $\psi(x;s)$ such that for all $M\models T$ and any enrichment $M_{U,V}$ of $M$ to $\LL_{U,V}$, we have:
\[\text{if }U\text{ and }V\text{ are externally }\phi\text{-separable, then }M_{U,V}\models \theta_{U,V}\]
and
\[\text{if }M_{U,V}\models \theta_{U,V} \text{, then }U\text{ and }V\text{ are externally }\psi\text{-separable.}\]
\end{proposition}

\begin{proof}
Let $k_{1}$ be the VC-dimension of $\phi(x;t)$. By the dual version of the $(p,q)$-theorem (see \cite{Mat} and \cite[Corollary\,6.13]{SimBook}) there exists $q_{1}$ and $n_{1}$ such that for any set $X$, any finite $A\subseteq X$ and any $\mathcal{S}\subseteq\Part{X}$ of VC-dimension at most $k_{1}$, if for all $A_{0}\subseteq A$ of size at most $q_{1}$ there exist $S\in\mathcal{S}$ containing $A_{0}$, then there exists $S_{1}\ldots S_{n_{1}} \in \mathcal{S}$ such that $A\subseteq\bigcup_{i\leq n_{1}} S_{i}$. Let $k_{2}$ be the VC-dimension of $\bigcup_{i=1}^{n_{1}}\phi(x;t_{i})$ and $q_{2}$ and $n_{2}$ the bounds obtained by the dual $(p,q)$-theorem for families of VC-dimension at most $k_{2}$. Let
\[\theta_{U,V} := \forall x_{1}\ldots x_{q_{1}},y_{1}\ldots y_{q_{2}}\,\bigwedge_{i\leq q_{1}} U(x_{i})\wedge\bigwedge_{j\leq q_{2}} V(y_{j}) \imp \exists t\,\bigwedge_{i\leq q_{1}}\phi(x_{i};t)\wedge\bigwedge_{i\leq q_{2}}\neg\phi(y_{j};t).\]

Now,  let $M\prec^+ N\models T$, $U$ and $V$ be subsets of $M^{|x|}$ and $d\in N$ be a tuple. If $U\subseteq \phi(M;d)$ and $V\subseteq\neg\phi(M;d)$ then for any $A\subseteq U$ and $B\subseteq V$ finite there exists $d_{0}\in M$ such that $A\subseteq \phi(M;d_{0})$ and $B\subseteq \neg\phi(M;d_{0})$. In particular, $M_{U,V}\models \theta_{U,V}$.

Suppose now that $M_{U,V}\models\theta_{U,V}$. Let $B_{0}\subseteq V$ have cardinality at most $q_{2}$. The family $\{\phi(M;d)\mid d\in M$ a tuple and $B_{0}\subseteq\neg\phi(M;d)\}$ has VC-dimension at most $k_{1}$ (a subfamily always has lower VC-dimension). Because $M_{U,V}\models\theta_{U,V}$, for any $A_{0}\subseteq U$ of size at most $q_{1}$, there exists $d\in M$ such that $A_{0} \subseteq \phi(M;d)$ and $B_{0}\subseteq \neg\phi(M;d)$. It follows that for any finite $A\subseteq U$ there are tuples $d_{1}\ldots d_{n_{1}}\in M$ such that $A\subseteq \bigvee_{i\leq n_{1}}\phi(M;d_{i})$ and for all $i\leq n_{1}$, $B_{0}\subseteq\neg\phi(M;d_{i})$, in particular, $B_{0}\subseteq\neg(\bigvee_{i\leq n_{1}}\phi(M;d_{i}))$. By compactness, there exists tuple $d_{1}\ldots d_{n_{1}}\in N$ such that $U\subseteq \bigvee_{i\leq n_{1}}\phi(M;d_{i})$ and $B_{0}\subseteq\neg(\bigvee_{i\leq n_{1}}\phi(M;d_{i}))$.

The family $\{\neg(\bigvee_{i\leq n_{1}}\phi(M;d_{i}))\mid d_{i}\in N$ tuples and $U\subseteq\bigvee_{i\leq n_{1}}\phi(M;d_{i})\}$ has VC-dimension at most $k_{2}$. We have just shown that for any $B_{0}$ of size at most $q_{2}$, there is an element of that family containing $B_{0}$. It follows by the $(p,q)$-property and compactness that there exists tuples $d_{i,j}\in N\supsel M$ such that $V \subseteq \bigvee_{j\leq n_{2}}\neg(\bigvee_{i\leq n_{1}}\phi(M;d_{i,j})) = \neg(\bigwedge_{j\leq n_{2}}\bigvee_{i\leq n_{1}}\phi(M;d_{i,j}))$ and $U \subseteq \bigwedge_{j\leq n_{2}}\bigvee_{i\leq n_{1}}\phi(M;d_{i,j})$. Hence $U$ and $V$ are externally $\bigwedge_{j\leq n_{2}}\bigvee_{i\leq n_{1}}\phi(x;t_{i,j})$-separable.
\end{proof}

We would now like to characterise enrichments $\tT$ of $\NIP$ theories that do not add new externally separable definable sets, i.e. $\tL$-definable sets that are externally $\LL$-separable but not internally $\LL$-separable. We show that if there is one model of $\tT$ where this property holds uniformly, then it holds in all models of $T$.

\begin{proposition}[no new sep]
Let $T$ be an $\NIP$ $\LL$-theory (with at least two constants), $\tL\supseteq\LL$ be some language, $\tT\supseteq T$ be a complete $\tL$-theory and $\chi_{1}(x;s)$ and $\chi_{2}(x;s)$ be $\tL$-formulas. The following are equivalent:
\begin{thm@enum}
\item\label{cond:ext sep} For all $\LL$-formulas $\phi(x;t)$, all $M\models\tT$ and all $a\in M$ there exists an $\LL$-formula $\xi(x;z)$ such that if $\chi_{1}(M;a)$ and $\chi_{2}(M;a)$ are externally $\phi$-separated then they are $\xi$-separated;
\item For all $\LL$-formulas $\phi(x;t)$, there exists an $\LL$-formula $\xi(x;z)$ such that for all $M\models\tT$ and all $a\in M$, if $\chi_{1}(M;a)$ and $\chi_{2}(M;a)$ are externally $\phi$-separated then they are $\xi$-separated;
\item\label{cond:nice model} For all $\LL$-formulas $\phi(x;t)$, there exists an $\LL$-formula $\xi(x;z)$ and $M\models\tT$ such that for all $a\in M$, if $\chi_{1}(M;a)$ and $\chi_{2}(M;a)$ are externally $\phi$-separated then they are $\xi$-separated.
\end{thm@enum}
\end{proposition}

\begin{proof}
The implications (ii) $\imp$ (i) and (ii) $\imp$ (iii) are trivial.

Let us now show that (iii) implies (ii). By Proposition\,\ref{prop:ext sep is def}, there exists an $\tL$-formula $\theta(s)$ and an $\LL$-formula $\psi(x;u)$ such that for all $N\models\tT$ and $a\in N$:
\[\chi_{1}(N;a)\text{ and }\chi_{2}(N;a)\text{ externally }\phi\text{-separated implies } N\models \theta(a)\]
and
\[N\models \theta(a) \text{ implies }\chi_{1}(N;a)\text{ and }\chi_{2}(N;a)\text{ externally }\psi\text{-separated.}\]
Let $M$ and $\xi$ be as in condition (iii) with respect to $\psi$. We have:
\[M\models \forall s\,\theta(s)\imp \exists u\,(\forall x\,(\chi_{1}(x;s)\imp\xi(x;u))\wedge(\chi_{2}(x;s)\imp\neg\xi(x;u))).\]
As $\tT$ is complete, this must hold in any $N\models\tT$. Thus, if $\chi_{1}(N;a)$ and $\chi_{2}(N;a)$ are externally $\phi$-separated, we have $N\models\theta(a)$ and hence $\chi_{1}(N;a)$ and $\chi_{2}(N;a)$ are $\xi$-separated.

There remains to prove that (i) $\imp$ (iii). Pick any $M \prec^{+}\fU\models\tT$. By (i), it is impossible to find, in any elementary extension $(\fU^{\star},M^{\star})$ of the pair $(\fU,M)$, a tuple $a\in M^{\star}$ and $b\in\fU^{\star}$ such that $\chi_{1}(M^{\star};a)$ and $\chi_{2}(M^{\star};a)$ are separated by $\phi(M^{\star};b)$, but they are not separated by any set of the form $\xi(M^{\star};c)$ where $\xi$ is an $\LL$-formula and $c\in M^{\star}$. By compactness, there exists $\xi_{i}(x;u_{i})$ for $i\leq n$ such that for all $a\in M$ if $\chi_{1}(M;a)$ and $\chi_{2}(M;a)$ are externally $\phi$-separated, there exists an $i$ such that they are $\xi_{i}$-separated. By classic coding tricks, we can ensure that $i=1$.
\end{proof}

\begin{definition}(Uniform stable embeddedness)
Let $M$ be an $\LL$-structure and $A\subseteq M$. We say that $A$ is uniformly stably embedded in $M$ if for all formulas $\phi(x;t)$ there exists a formula $\chi(x;s)$ such that for all tuples $b\in M$ there exists a tuple $a\in A$ such that $\phi(A,b) = \chi(A,a)$.
\end{definition}

\begin{remark}[strong 3]
If there exists $M\models\tT$ such that $\restr{M}{\LL}$ is uniformly stably embedded in every elementary extension, then such an $M$ witnesses Condition\,\ref{cond:nice model} for every choice of formulas $\chi_{1}$ and $\chi_{2}$.
\end{remark}

\begin{corollary}[def imp def]
Let $T$ be an $\NIP$ $\LL$-theory that eliminates imaginaries, $\tL\supseteq\LL$ be some language and $\tT\supseteq T$ be a complete $\tL$-theory. Suppose that there exists $M\models \tT$ such that $\restr{M}{\LL}$ is uniformly stably embedded in every elementary extension. Let $\phi(x;t)$ be an $\LL$-formula, $N\models \tT$, $A = \dcl[\tL]<\eqop>(A)\subseteq \eq{N}$ and $p\in\TP[x]<\phi>(N)$. If $p$ is $\tLeq(A)$-definable, then it is in fact $\LL(\Real(A))$-definable where $\Real$ denotes the set of all $\LL$-sorts.
\end{corollary}

\begin{proof}
Let $a\models p$. Then $X := \{m\in N\mid \phi(x;m)\in p\} = \{m\in {N\mid} \models\phi(a;m)\}$ is $\LL$-externally definable and $\tLeq(A)$-definable (by some $\tL$-formula $\chi$). It follows from Remark\,\ref{rem:strong 3} that Condition\,\ref{cond:nice model} holds and hence, by Condition\,\ref{cond:ext sep}, taking $\chi_{1} = \chi$ and $\chi_{2} = \neg\chi$, it follows that $X$ is $\LL$-definable.

Because $T$ eliminates imaginaries, we have just shown that we can find $\code{X}[\LL]\in \Real$. But $X$ is also $\tLeq(A)$-definable, hence any $\tLeq(A)$-automorphism of $\eq{N}$ stabilises $X(N)$ globally and therefore fixes $\code{X}[\LL]$. If we assume that $N$ is strongly $|A|^+$-homogeneous (and we can), it follows that $\code{X}[\LL]\in \dcleq[\tL](A) = A$. Thus $\code{X}[\LL]\in A\cap \Real  = \Real(A)$ and $X$ is $\LL(\Real(A))$-definable.
\end{proof}

We will need the following result, which is \cite[Proposition\,2.11]{SimInvNIP}.

\begin{proposition}[local inv]
Let $T$ be any theory, $\phi(x;y)$ an NIP formula, $M\prec^{+} N \models T$ and $p(x)$ a global $M$-invariant $\phi$-type. Let $b,b'\in \fU \supsel N$ such that both $\tp(b/N)$ and $\tp(b'/N)$ are finitely satisfiable in $M$ and $\tp[\opp{\phi}](b/N)=\tp[\opp{\phi}](b'/N)$. Then we have $\restr{p}{\fU}\vdash \phi(x;b) \iff \phi(x;b')$.
\end{proposition}

\begin{proposition}[inv imp inv]
Let $T$ be an $\NIP$ $\LL$-theory, $\tL\supseteq\LL$ be some language and $\tT\supseteq T$ be a complete $\tL$-theory. Let $\Real$ denote the set of $\LL$-sorts. Suppose that there exists $M\models \tT$ such that $\restr{M}{\LL}$ is uniformly stably embedded in every elementary extension. Let $\phi(x;t)$ be an $\LL$-formula, $N\models \tT$ be sufficiently saturated, $A\subseteq N$ and $p\in\TP[x]<\phi>(N)$ be $\tL(A)$-invariant. Assume that every $\tL(A)$-definable set (in some Cartesian power of $\Real$) is consistent with some global $\LL(\Real(A))$-invariant type. Then $p$ is $\LL(\Real(A))$-invariant.
\end{proposition}

\begin{proof}
Let us first assume that $A\models\tT$. Let $b_{1}$ and $b_{2}$ be such $p(x)\vdash \phi(x,b_{1})\wedge \neg\phi(x,b_{2})$. We have to show that $\tp[\LL](b_{1}/A)\neq\tp[\LL](b_{2}/A)$. Let $p_{i} = \tp[\tL](b_{i}/A)$, $\Sigma(t)$ be the set of $\tL(N)$-formulas $\theta(t)$ such that $\neg\theta(A) = \emptyset$ and $\Delta(t_{1},t_{2})$ be the set:
\[p_{1}(t_{1})\cup p_{2}(t_{2})\cup\Sigma(t_{1})\cup\Sigma(t_{2})\cup\{\phi(n,t_{1})\iff\phi(n,t_{2})\mid n\in N\}.\]

If $\Delta$ were consistent, there would exist $b_{1}^{\star}$ and $b_{2}^{\star}$ such that $b_{i}\typeq{\tL(A)} b_{i}^{\star}$, $\tp[\tL](b_{i}^{\star}/N)$ is finitely satisfiable in $A$ and $\tp[\opp{\phi}](b_{1}^{\star}/N) = \tp[\opp{\phi}](b_{2}^{\star}/N)$. Applying Proposition\,\ref{prop:local inv} it would follow that $p(x)\vdash \phi(x;b_{1}^{\star})\iff\phi(x;b_{2}^{\star})$. But, because $p$ is $\tL(A)$-invariant and $p(x)\vdash \phi(x,b_{1})\wedge \neg\phi(x,b_{2})$, we also have that $p(x)\vdash \phi(x;b_{1}^{\star})\wedge\neg\phi(x;b_{2}^{\star})$, a contradiction.

By compactness, there exists $\psi_{i} \in p_{i}$, $\theta_{i}\in\Sigma$, $n\in\omega$ and $(c_{i})_{i\in n}\in N$ such that \[\forall t_{1},t_{2}\,\theta_{1}(t_{1})\wedge\theta_{2}(t_{2})\wedge(\bigwedge_{i}\phi(c_{i},t_{1})\iff\phi(c_{i},t_{2}))\wedge\psi_{1}(t_{1}) \imp \neg\psi_{2}(t_{2}).\]
In particular, because $\neg\theta_{i}(A) = \emptyset$, for all $m_{1}$ and $m_{2}\in A$, $(\bigwedge_{i}\phi(c_{i},m_{1})\iff\phi(c_{i},m_{2}))\wedge\psi_{1}(m_{1}) \imp \neg\psi_{2}(m_{2})$. For all $\epsilon : n \to 2$, let $\phi_{\epsilon}(t,c) := \bigwedge_{i}\phi(c_{i},t)^{\epsilon(i)}$ where $\phi^{1} = \phi$ and $\phi^{0} = \neg\phi$. It follows that if $\phi_{\epsilon}(A,c)\cap\psi_{1}(A)\neq\emptyset$, then $\phi_{\epsilon}(A,c)\cap\psi_{2}(A) = \emptyset$. Let
\[\theta(t,c) := \bigvee_{\phi_{\epsilon}(A,c)\cap\psi_{1}(A)\neq\emptyset} \phi_{\epsilon}(c,t).\] We have $\psi_{1}(A)\subseteq \theta(A,c)$ and $\psi_{2}(A)\cap\theta(A,c) = \emptyset$, i.e. $\psi_{1}(A)$ and $\psi_{2}(A)$ are externally $\theta$-separable. By Proposition\,\ref{prop:no new sep} and Remark\,\ref{rem:strong 3}, $\psi_{1}(A)$ and $\psi_{2}(A)$ are in fact $\xi$-separable for some $\LL(\Real(A))$-formula $\xi$. It follows that $N\models\forall t_{1},t_{2}\,(\psi_{1}(t_{1})\imp\xi(t_{1}))\wedge(\psi_{2}(t_{2})\imp\neg\xi(t_{2}))$ and, in particular $N\models \xi(b_{1})\wedge\neg\xi(b_{2})$. So $\tp[\LL](b_{1}/A)\neq\tp[\LL](b_{2}/A)$.\smallskip

Let us now conclude the proof when $A$ is not a model. Let $M\models\tT$ contain $A$ and pick any $a$ and $b\in N$ such that $a\typeq{\LL(\Real(A))} b$.

\begin{claim}
There exists $M^{\star}\typeq{\tL(A)} M$ (in particular it is a model of $\tT$ containing $A$) such that $a\typeq{\LL(\Real(M^{\star}))} b$.
\end{claim}

\begin{proof}
By compactness, it suffices, given $\chi(y,z) \in \tp[\tL](M/A)$, where $y$ is a tuple of $\Real$-variables, and $\psi_{i}(t;y)$ a finite number of $\LL$-formulas, to find tuples $m$, $n$ such that $\models\chi(m,n)\wedge\bigwedge_{i}\psi(a;m)\iff\psi(b;m)$. By hypothesis on $A$, there exists $q\in\TP[y](\restr{N}{\LL})$ which is $\LL(\Real(A))$-invariant and consistent with $\exists z\,\chi(y,z)$. Let $m\models\restr{q}{\Real(A)ab}\cup\{\chi(y)\}$. Then $\tp[\LL](a/m) = \tp[\LL](b/m)$ and $\models \exists z\,\chi(m,z)$. In particular, we can also find $n$.
\end{proof}

As $p$ is $\tL(A)$-invariant it is in particular $\tL(M^{\star})$-invariant. But, as shown above, $p$ is then $\LL(\Real(M^{\star}))$-invariant. It follows that $p \vdash \phi(x;a)\iff\phi(x;b)$.
\end{proof}

The assumption that all $\tL(A)$-definable sets are consistent with some global $\LL(A)$-invariant type may seem like a surprising assumption. Nevertheless, considering a coheir (in the sense of $\tT$, whose restriction to $\LL$ is also a coheir in the sense of $T$), this assumption always holds when $A$ is a model of $\tT$.

\section{Valued differential fields}

The main motivation for the results in the previous sections was to understand definable and invariant types in valued differential fields and more specifically those with a contractive derivation, i.e. for all $x$, $\val(\D(x))\geq\val(x)$. In \cite{ScaDValF}, Scanlon showed that the theory of valued fields with a valuation preserving derivation has a model completion named $\VDF$. It is the theory of $\D$-Henselian fields whose residue field is a model of $\DCF[0]$, whose value group is divisible and such that for all $x$ there exists a $y$ with $\D(y) = 0$ and $\val(y) = \val(x)$.

The main result that we will be needing here is that the theory $\VDF$ eliminates quantifiers in the one sorted language $\LDdiv$ consisting of the language of rings enriched with a symbol $\D$ for the derivation and a symbol $x\Div y$ interpreted as $\val(x)\leq\val(y)$. This result implies that for all substructures $A \substr M\models \VDF$ the map sending $p = \tp[\LDdiv](c/A)$ to $\tpprol[\omega]{p} := \tp[\Ldiv]((\D^{i}(c))_{i\in\omega}/A)$ is injective, where $\Ldiv := \LDdiv\sminus\{\D\}$ denotes the one sorted language of valued fields.

\begin{lemma}[nice model VDF]
Let $k\models\DCF[0]$. The Hahn field $k((t^{\Rr}))$, with derivation $\D(\sum_i a_it^i) = \sum_i \D(a_i)t^i$ and its natural valuation, is a models of $\VDF$ and its reduct to $\Ldiv$ is  uniformly stably embedded in every elementary extension.
\end{lemma}

\begin{proof}
The fact that $k((t^{\Rr}))\models\VDF$ follows from the fact that its residue field $k$ is a model of $\DCF[0]$, its value group $\Rr$ is a divisible ordered Abelian group and that Hahn fields are spherically complete, cf. \cite[Proposition\,6.1]{ScaDValF}.

The fact that $k((t^{\Rr}))$ is uniformly stably embedded in every elementary extension is shown in \cite[Corollary A.7]{RidVDF}.
\end{proof}

Recall that Haskell, Hrushovski and Macpherson \cite{HasHruMacACVF} showed that algebraically closed valued fields eliminate imaginaries provided the geometric sorts are added. We will be denoting by $\Geom$ the set of all geometric sorts.

\begin{proposition}[def imp def VDF]
Let $A = \acleq[\LDdiv](A) \subseteq M\models\VDF$. A type $p\in\TP<\Ldiv>(M)$ is $\LDdiveq(A)$-definable if and only if it is $\Ldiveq(\Geom(A))$-definable.
\end{proposition}

\begin{proof}
If $p$ is $\Ldiveq(\Geom(A))$-definable then it is in particular $\LDdiveq(A)$-definable. The reciprocal implication follows immediately from Corollary\,\ref{cor:def imp def} and Lemma\,\ref{lem:nice model VDF}.
\end{proof}

An immediate corollary of this proposition is an elimination of imaginaries result for canonical bases of definable types in $\VDF$:

\begin{corollary}[equiv def]
Let $A = \acleq[\LDdiv](A) \subseteq M\models\VDF$ and $p\in\TP<\LDdiv>(M)$. The following are equivalent:
\begin{thm@enum}
\item $p$ is $\LDdiveq(A)$-definable;
\item $\tpprol[\omega](p)$ is $\Ldiveq(\Geom(A))$-definable;
\item $p$ is $\LDdiveq(\Geom(A))$-definable.
\end{thm@enum}
\end{corollary}

\begin{proof}
The implication (iii) $\imp$ (i) is trivial. Let us now assume (i). An $\Ldiv(M)$-formula $\phi(\uple{x};m)$ is in $\tpprol[\omega](p)$ if and only if $\phi(\prol[\omega](x);m) \in p$, where $\prol[\omega](x) = (\D^i(x))_{i\in\omega}$. It follows that $\tpprol[\omega](p)$ is $\LDdiveq(A)$-definable. By Proposition\,\ref{prop:def imp def VDF}, $\tpprol[\omega](p)$ is in fact $\Ldiveq(\Geom(A))$-definable.

Let us now assume (ii) and let $\psi(x;m)$ be any $\LDdiv(M)$-formula. By quantifier elimination, $\psi(x;m)$ is equivalent to $\phi(\prol[\omega](x);\prol[\omega](m))$ for some $\Ldiv$-formula $\phi(\uple{x};\uple{t})$. Therefore $\psi(x;m)\in p$ if and only if $\phi(\uple{x};\prol[\omega](m))\in\tpprol[\omega](p)$ and hence $p$ is $\LDdiveq(\Geom(A))$-definable.
\end{proof}

In \cite{RidVDF}, it is shown that there are enough definable types to use this partial elimination of imaginaries result to obtain elimination of imaginaries to the geometric sorts for $\VDF$.

Thanks to the result in Section\,\ref{sec:ext sep} and results from \cite{RidVDF}, we can also characterise invariant types in $\VDF$. Note that, although the main results in \cite{RidVDF} depend on the results proved in the present paper, the result from \cite{RidVDF} that we will be using in what follows does not.

\begin{proposition}[inv imp inv VDF]
Let $M\models\VDF$ and $A = \acleq[\LDdiv](A) \subseteq \eq{M}$. A type $p\in\TP<\Ldiv>(M)$ is $\LDdiveq(A)$-invariant if and only if it is $\Ldiveq(\Geom(A))$-invariant.
\end{proposition}

\begin{proof}
To prove the non obvious implication, by Proposition\,\ref{prop:inv imp inv}, we have to show that $\VDF$ has a model whose underlying valued field is uniformly stably embedded in any elementary extension --- that is tackled in Lemma\,\ref{lem:nice model VDF} --- and that any $\LDdiveq(A)$-definable set (in the sort $\K$) is consistent with an $\Ldiveq(\Geom(A))$-invariant $\Ldiv$-type. It follows from \cite[Proposition\,9.7]{RidVDF} (applied to $T = \ACVF$ and $\tT = \VDF$) that any $\LDdiveq(A)$-definable set (in the sort $\K$) is consistent with an $\LDdiveq(A)$-definable $\Ldiv$-type. But, by Proposition\,\ref{prop:def imp def VDF}, such a type is $\Ldiveq(\Geom(A))$-definable.
\end{proof}

\begin{corollary}[equiv inv]
Let $A = \acleq[\LDdiv](A) \subseteq M\models\VDF$ and $p\in\TP<\LDdiv>(M)$. The following are equivalent:
\begin{thm@enum}
\item $p$ is $\LDdiveq(A)$-invariant;
\item $\tpprol[\omega](p)$ is $\Ldiveq(\Geom(A))$-invariant;
\item $p$ is $\LDdiveq(\Geom(A))$-invariant.
\end{thm@enum}
\end{corollary}

\begin{proof}
This is proved as in Corollary\,\ref{cor:equiv def}, except that Proposition\,\ref{prop:inv imp inv VDF} is used instead of Proposition\,\ref{prop:def imp def VDF}.
\end{proof}

We can now give a characterisation of forking in $\VDF$.

\begin{corollary}
Let $M\models\VDF$ be $\card{A}^{+}$-saturated, $A = \acleq[\LDdiv](A) \subseteq M$ and $\phi(x)$ be an $\LDdiv(M)$-formula. Then $\phi(x)$ does not fork over $A$ if and only if for all $\Ldiv(M)$-formulas such that $\phi(x)$ is equivalent to $\psi(\prol[\omega](x))$, $\psi(\uple{x})$ does not fork over $\Geom(A)$ (in $\ACVF$).
\end{corollary}

\begin{proof}
Let us first assume that $\phi(x)$ does not fork over $A$ and let $p$ be a global non forking extension of $\phi(x)$. As $\VDF$ is $\NIP$, by \cite[Proposition\,2.1]{HruPilNip}, $p$ is invariant under all automorphisms that fix Lascar strong type over $A$. But, because $\VDF$ has the invariant extension property (cf. \cite[Theorem\,2.14]{RidVDF}), Lascar strong type and strong type coincide in $\VDF$ (see \cite[Proposition\,2.13]{HruPilNip}), hence $p$ is $\LDdiveq(A)$-invariant. It follows from Corollary\,\ref{cor:equiv inv} that $\tpprol[\omega](p)$ is $\Ldiveq(\Geom(A))$-invariant and hence $\psi(\uple{x})$ does not fork over $\Geom(A)$.

Let us now assume that no $\psi(\uple{x})$ such that $\phi(x)$ is equivalent to $\psi(\prol[\omega](x))$ forks over $\Geom(A)$. Then there exists $q\in\TP[\uple{x}]<\Ldiv>(M)$ which is $\Ldiveq(\Geom(A))$-invariant and consistent with all such formulas $\psi(\uple{x})$. Now, the image of the continuous map $\tpprol[\omega] : \TP[x]<\LDdiv>(M)\to\TP[\uple{x}]<\Ldiv>(M)$ is closed and if $\chi(\uple{x})$ is an $\Ldiv(M)$-formula containing the image of $\tpprol[\omega]$ and $\psi(\uple{x})$ is as above, $\chi(\prol[\omega](x))\wedge\psi(\prol[\omega](x))$ is also equivalent to $\phi(x)$. Therefore, $q = \tpprol[\omega](p)$ for some $\LDdiveq(A)$-invariant $p\in\TP[x]<\LDdiv>(M)$. This type $p$ implies $\phi(x)$ and hence $\phi(x)$ does not fork over $A$.
\end{proof}

\begin{remark}
The previous corollary is somewhat unsatisfying as one needs to consider all possible ways of describing $\phi(x)$ as the prolongation points of an $\Ldiv$-formula $\psi$ (with parameters in a saturated model) to conclude whether $\phi$ forks or not.

Considering only one such $\psi$ cannot be enough. For example, consider any definable set $\phi(x)$ forking (in $\VDF$) over $A$ and let $\psi(x_{0},x_{1})  = (\val(x_{0})\geq 0 \wedge \val(x_{1}) < 0) \vee \phi(x_{0})$. Then the set $\{x\in M\mid M\models\psi(x,\D(x))\} = \phi(M)$ but $\psi$ does not fork over $A$ (in $\ACVF$). The obstruction here might seem frivolous, but it is the core of the problem. Indeed, it is not clear if there is a way, given $\phi$ to find a formula $\psi$ as above that does not contain "large" subsets with no prolongation points.
\end{remark}

\printbibli
\end{document}